\documentclass[a4paper, 11pt]{amsart}
\title[Global critical chart]
{Global critical chart for local Calabi--Yau threefolds}
\date{}
\author{Tasuki Kinjo and Naruki Masuda}

\address{Johns Hopkins University, 3400 N. Charles St., Baltimore, MD, USA}
\email{nmasuda2@jh.edu}

\address{graduate school of mathematical science, the university of tokyo, 3-8-1 komaba,
meguroku, tokyo 153-8914, japan.}
\email{tasuki.kinjo@ipmu.jp}

\makeatletter
 
  \@addtoreset{equation}{section}
 \makeatother

\usepackage{amsmath, amssymb, amsthm, amscd, comment, mathtools, color}
\usepackage{stmaryrd}
\usepackage{todonotes}
\usepackage[frame,cmtip,curve,arrow,matrix,line,graph]{xy}
\usepackage[mathscr]{euscript}
\usepackage{mathrsfs}
\usepackage{tikz}
  \usepackage{tikz-cd}
\usetikzlibrary{intersections, calc}

\usepackage[T1]{fontenc}
\usepackage{hyperref}

\newtheorem{thm}{Theorem}[section]
\newtheorem{prop}[thm]{Proposition}
\newtheorem{def-prop}[thm]{Definition-Proposition}
\newtheorem{lem}[thm]{Lemma}
\newtheorem{lemdef}[thm]{Lemma-Definition}
\newtheorem{cor}[thm]{Corollary}
\newtheorem*{thm*}{Theorem}

\theoremstyle{definition}
\newtheorem{defin}[thm]{Definition}

\newtheorem*{NaC}{Notation and Convention}
\newtheorem*{ACK}{Acknowledgement}

\newtheorem{rmk}[thm]{Remark}

\newtheorem{ex}[thm]{Example}

\DeclareMathOperator{\Spec}{Spec}
\DeclareMathOperator{\Sp}{Sp}

\newcommand{\bA}{\mathbb{A}}

\newcommand{\bC}{\mathbb{C}}

\newcommand{\bL}{\mathbb{L}}

\newcommand{\bQ}{\mathbb{Q}}

\newcommand{\mcO}{\mathcal{O}}

\newcommand{\mcU}{\mathcal{U}}

\DeclareMathOperator{\Hom}{Hom}
\DeclareMathOperator{\Tot}{Tot}

\DeclareMathOperator{\Ext}{Ext}

\DeclareMathOperator{\Sym}{Sym}

\DeclareMathOperator{\cl}{cl}

\DeclareMathOperator{\Perv}{Perv}

\DeclareMathOperator{\op}{op}

\DeclareMathOperator{\HC}{HC}
\DeclareMathOperator{\mH}{H}

\DeclareMathOperator{\HH}{HH}
\DeclareMathOperator{\Perf}{Perf}
\DeclareMathOperator{\Qcoh}{Qcoh}

\DeclareMathOperator{\CAlg}{CAlg}
\DeclareMathOperator{\Alg}{Alg}

\DeclareMathOperator{\Mod}{Mod}
\DeclareMathOperator{\RMod}{RMod}
\DeclareMathOperator{\LMod}{LMod}
\newcommand\BMod{\mathrm{BMod}}

\DeclareMathOperator{\Obj}{Obj}
\DeclareMathOperator{\Map}{Map}

\DeclareMathOperator{\Tr}{Tr}

\DeclareMathOperator{\ord}{ord}

\DeclareMathOperator{\dg}{dg}
\DeclareMathOperator{\dgcat}{dgcat}
\DeclareMathOperator{\Crit}{Crit}

\newcommand\DCrit{\mathop{\mathbf{Crit}}\nolimits}

\newcommand\ddr{d_{\mathrm{dR}}}

\newcommand{\cA}{{\mathcal A}}

\newcommand{\cO}{{\mathcal O}}
\newcommand{\cP}{{\mathcal P}}

\newcommand{\fM}{{\mathfrak M}}

\newcommand\bs{\boldsymbol}

\newcommand{\pbcorner}[1][ul]{\save*!/#1+1.6pc/#1:(1,-1)@^{|-}\restore}
\newcommand{\pocorner}[1][dr]{\save*!/#1+1.6pc/#1:(1,-1)@^{|-}\restore}

\newcommand{\eA}{{\EuScript A}}

\newcommand{\eC}{{\EuScript C}}
\newcommand{\eD}{{\EuScript D}}
\newcommand{\eE}{{\EuScript E}}
\newcommand{\eP}{{\EuScript P}}
\newcommand{\eS}{{\EuScript S}}

\newcommand\Cat{\mathop{\eC \mathrm{at}_{\infty}}}
\newcommand\PrL{\mathop{\eP \mathrm{r}^L}}
\newcommand\PrLc{\mathop{\eP \mathrm{r}^L_{\omega}}}
\newcommand\PrRc{\mathop{\eP \mathrm{r}^R_{\omega}}}
\newcommand\PrLst{\mathop{\eP \mathrm{r}^L_{\mathrm{st}}}}
\newcommand\PrLstc{\mathop{\eP \mathrm{r}^L_{\mathrm{st}, \omega}}}
\newcommand\PrR{\mathop{\eP \mathrm{r}^R}}
\newcommand\Catk{\mathop{\mathrm{LinCat}_k^{\mathrm{St}}}}
\newcommand\Catkcg{\mathop{\mathrm{LinCat}_k^{\mathrm{St}, \omega}}}
\DeclareMathOperator{\Fun}{Fun}
\DeclareMathOperator{\Mnd}{Mnd}
\DeclareMathOperator{\LFun}{LFun}
\DeclareMathOperator{\Id}{Id}
\DeclareMathOperator{\ev}{ev}
\DeclareMathOperator{\coev}{coev}

\begin{document}

\begin{abstract}
In this paper, we investigate Keller's deformed Calabi--Yau completion of the derived category of coherent sheaves on a smooth variety.
In particular, for an $n$-dimensional smooth variety $Y$, we describe the derived category of the total space of an $\omega_Y$-torsor as a certain deformed $(n+1)$-Calabi--Yau completion of the derived category of $Y$.

As an application, we investigate the geometry of the derived moduli stack of compactly supported coherent sheaves on a local curve, i.e., a Calabi--Yau threefold of the form $\mathrm{Tot}_C(N)$, where $C$ is a smooth projective curve and $N$ is a rank two vector bundle on $C$.
We show that the derived moduli stack is equivalent to the derived critical locus of a function on a certain smooth moduli space.
This result will be used by the first author and Naoki Koseki in their joint work on Higgs bundles and Gopakumar--Vafa invariants.
\end{abstract}

\maketitle

\setcounter{tocdepth}{1}
\tableofcontents

\section{Introduction}

\subsection{Motivations}

The theory of dg-category can be regarded as a non-commutative version of derived algebraic geometry.
Indeed, we can regard a derived scheme as a non-commutative space by considering its derived category of perfect complexes.
As in usual algebraic geometry, Calabi--Yau structure plays an important role in dg-category theory.
In the study of Calabi--Yau dg-category, Keller \cite{Kel11} introduced the  Calabi--Yau completion of a dg-category and its deformed version.
More precisely, for a given small dg-category $\eA$, an integer $n$ and a negative cyclic class $c \in \HC^{-}_{n-1}(\eA)$, he defined a new small dg-category
\[
\Pi_{n + 1}(\eA, c) \in \dgcat
\]
called the $(n+1)$-deformed Calabi--Yau completion of $\eA$.
It carries a natural left $(n + 1)$-Calabi--Yau structure in the sense of Brav--Dyckerhoff's paper \cite{BDI}.

It is shown in \cite{IQ18} that the undeformed Calabi--Yau completion is a non-commutative generalization of the total space of the canonical bundle on a smooth variety. In this paper, we investigate the geometric meaning of the \emph{deformed} Calabi--Yau completion.

\subsection{Results}

Let $Y$ be a quasi-projective smooth variety of dimension $n$ over the complex number field.
Take a cohomology class $c \in \mH^1(Y, \omega_Y)$.
Let $X$ be the total space of the $\omega_Y$-torsor corresponding to $c$. 

\begin{thm}[= Theorem \ref{thm:defCYcomp}]\label{thm:intromain}
Let $\Perf(Y)$ (resp. $\Perf(X)$) be the derived dg-category of perfect complexes on $Y$ (resp. $X$).
Regard $c$ as an $(n-1)$-th Hochschild homology class using the natural inclusion
$\mH^1(Y, \omega_Y) \hookrightarrow \HH_{n-1}(\Perf(Y)) $.
One can show that $c$ naturally lifts to a negative cyclic class $\tilde{c}$.
Then we have a natural Morita equivalence of dg-categories
\[
\Perf(X) \simeq \Pi_{n+1}(\Perf(Y),  \tilde{c})
\]
preserving the natural left $n+1$-Calabi--Yau structure.
\end{thm}

The above theorem can be applied in the following situation.
Let $C$ be a quasi-projective smooth projective curve and $N$ be a vector bundle over $Z$ of rank two such that there exists an isomorphism 
$\det(N) \cong \omega_C$.
Assume that we are given a short exact sequence of locally free sheaves on $C$
\[
0 \to L_1 \to N \to L_2 \to 0
\]
such that $L_1$ and $L_2$ are line bundles.
Then $\Tot_{C}(N)$ is the total space of an $\omega_{\Tot_{C}(L_2)}$-torsor on $\Tot_{C}(L_2)$.
Therefore Theorem \ref{thm:intromain} implies that
$\Perf(\Tot_{C}(N))$ is Morita equivalent to a certain deformed $3$-Calabi--Yau completion of $\Perf(\Tot_{C}(L_2))$.
Combining this result and a theorem due to Bozec--Calaque--Scherotzke \cite[Corollary 6.19]{BCS20}, we obtain the following Corollary:

\begin{cor}[={Theorem \ref{thm:crit}}]\label{cor:intro}
Let $\bs{\fM}_{\Tot_{C}(N)}$ (resp. $\bs{\fM}_{\Tot_{C}(L_2)}$) be the derived moduli stack of compactly supported coherent sheaves on $\Tot_{C}(N)$ (resp. $\Tot_{C}(L_2)$).
Then there exists a regular function $f \in \Gamma(\bs{\fM}_{\Tot_{C}(L_2)}, \mcO_{\bs{\fM}_{\Tot_{C}(L_2)}})$ such that there exists a natural equivalence of $-1$-shifted symplectic derived Artin stacks
\[
\bs{\fM}_{\Tot_{C}(N)} \simeq \DCrit(f).
\]
\end{cor}

\subsection{Applications to Cohomological Donaldson--Thomas theory}

Corollary \ref{cor:intro} has a striking application in cohomological Donaldson--Thomas (CoDT) theory. To explain this, we briefly review CoDT theory here.

Donaldson--Thomas (DT) theory studies DT invariants of a  Calabi--Yau threefold $X$ which virtually count semistable sheaves on $X$. More precisely, for a homology class $\gamma \in H_*(X)$, we can define the Donaldson--Thomas invariant
\[
\mathrm{DT}_{\gamma}(X) \in \bQ
\]
which virtually counts semistable sheaves with the Chern character $\gamma$.
Assume that all semistable sheaves on $X$ with the Chern character $\gamma$ is stable and we are given an orientation (in the sense of \cite[Definition 2.57]{Joy15}) of the moduli stack $\fM_{X, \gamma}$ of stable sheaves on $X$ with the Chern character $\gamma$.
Then it follows from \cite[Theorem 6.9]{BBDJS15} that there exists a perverse sheaf 
\[
\varphi_{M_{X, \gamma}} \in \Perv(M_{X, \gamma})
\]
on the good moduli space of stable sheaves on $X$ with the Chern character $\gamma$ such that we have an equality
\[
\mathrm{DT}_{\gamma}(X) = \sum_{i} (-1)^i \dim \mH^i(M_{X, \gamma}, \varphi_{M_{X, \gamma}}).
\]
For general $\gamma$ with a given orientation of the moduli stack $\fM_{X, \gamma}$ of semistable sheaves on $X$ with the Chern character $\gamma$, it also follows from \cite[Theorem 6.9]{BBDJS15} that there exists a perverse sheaf 
\[
\varphi_{\fM_{X, \gamma}} \in \Perv(\fM_{X, \gamma})
\]
which recovers the Donaldson--Thomas invariant $\mathrm{DT}_{\gamma}$.
CoDT theory aims to study the perverse sheaf $\varphi_{\fM_{X, \gamma}}$.
At the time of this writing, almost nothing is known in CoDT theory for Calabi--Yau threefolds.

We can also define CoDT invariants for quivers with potentials whose dimension virtually count semistable representations of their Jacobi algebras.
In contrast to the Calabi--Yau threefold case, CoDT theory for quivers with potentials is well-developed (see e.g.  \cite{KS11,DM20}) and has a striking applications in representation theory (see e.g. \cite{Dav18}).
The main obstruction of CoDT theory for Calabi--Yau threefolds is that the moduli stack $\fM_{X, \gamma}$ does not necessarily have a global critical locus description, whereas the moduli stack of representations of a Jacobi algebra does.

Now let us return back to the statement of Corollary \ref{cor:intro}.
It states that the moduli stack of semistable sheaves on a local curve, i.e., Calabi--Yau threefold of the form 
\[
X = \Tot_C(N)
\]
where $C$ is a smooth projective curve has a global critical locus description.
This fact and several observations imply that the argument of Davison--Meinhardt's paper \cite{DM20} works for local curves and it is possible to establish a foundation of  CoDT theory for local curves.
When we have $N = \cO_C \oplus \omega_C$, the CoDT theory for $\Tot_C(N)$ is closely related to non-abelian Hodge theory (see \cite[\S 5]{Kin21}).
Therefore our result also have an application to non-abelian Hodge theory.
This point of view will be investigated by the first author and Naoki Koseki \cite{KK21}.

\subsection{Strategy of the proof}

The main point of the proof of Theorem \ref{thm:intromain} is to work with $k$-linear presentable stable $\infty$-categories rather than small dg-categories (e.g. we work with the stable $\infty$-category of quasi-coherent sheaves rather than the dg-category of perfect complexes).
We introduce a definition of the deformed Calabi--Yau completion of $k$-linear stable $\infty$-categories as the module category over a certain monad.
An advantage of this definition is that we can use the Barr--Beck--Lurie theorem to construct an equivalence between other $\infty$-categories.
For example, one can use the Barr--Beck--Lurie theorem to show that our definition of the deformed Calabi--Yau completion is compatible with Keller's original definition.
Theorem \ref{thm:intromain} is also a simple application of  the Barr--Beck--Lurie theorem.

\subsection{Relation to earlier works}

\begin{enumerate}
    \item In \cite[\S 3.5]{Hit19}, Hitchin studies the critical locus of a linear function on the Hitchin base in (a certain open subset of) the moduli space of stable $\mathrm{SL}(2, \bC)$-Higgs bundles. He shows that points in the critical locus corresponds to a one-dimensional sheaves on a certain non-compact Calabi--Yau threefold.
    
    Corollary \ref{cor:intro} can be regarded as a $\mathrm{GL}(n, \bC)$-version of this result.
    Our result is stated as an equivalence of $-1$-shifted symplectic derived Artin stacks, which is crucial for applications in Donaldson--Thomas theory.
    
    \item In \cite[Theorem 4.5]{MS21}, Maulik and Shen describe the moduli space of stable Higgs bundles on $C$ as the critical locus of a certain linear function of the Hitchin base for $L$-twisted Higgs bundles where $\deg L = 2g - 1$.
    
    Corollary \ref{cor:intro} recovers this result as follows:
    Let $N$ be the rank two vector bundle given by the non-trivial extension
    \[
    0 \to K_C \to N \to \cO_C \to 0.
    \]
    Then using Corollary \ref{cor:intro},
    one can show that the critical locus of Maulik--Shen's function is isomorphic to the moduli space of one-dimensional stable sheaves on $\Tot_C(N)$.
    Further, one can see that the closed immersion $\Tot_C(K_C) \hookrightarrow \Tot_C(N)$ induces an isomorphism of the moduli spaces of one-dimensional sheaves.
    Combining theses isomorphisms, we obtain the desired claim.
\end{enumerate}

\subsection{Structure of the paper}
The paper is organized as follows:

In Section \ref{sec:dg} we recall some basic definitions and results on $k$-linear stable $\infty$-categories.

In Section \ref{sec:defCY} we introduce the deformed Calabi--Yau completion of a $k$-linear stable $\infty$-categories.

In Section \ref{sec:main} we prove Theorem \ref{thm:intromain}.

In Section \ref{sec:app} we apply Theorem \ref{thm:intromain} to prove that the moduli stack of coherent sheaves on a local curve is described as a global critical locus.

\begin{ACK}

The first author is very grateful to Yukinobu Toda for suggesting the possibility that the moduli stacks of coherent sheaves on local curves are described as global critical loci.
He is also grateful to Naoki Koseki for the collaboration on the companion paper \cite{KK21}.
We thank Junliang Shen and Yukinobu Toda for their comments on the previous version of this article.

T.K. was supported by WINGS-FMSP
program at the Graduate School of Mathematical Science, the University of Tokyo and JSPS KAKENHI Grant number JP21J21118.

\end{ACK}

\begin{NaC}
\begin{itemize}

\item
Throughout the paper, we fix an algebraically closed field $k$ of characteristic zero.

\item
The $\infty$-category of spaces is denoted by $\eS$.

\item
The $\infty$-category of  $\infty$-categories is denoted by $\Cat$.

\item
For $\infty$-categories $\eC_1$ and $\eC_2$,
we let $\Fun(\eC_1, \eC_2)$ denote the functor $\infty$-category and $\LFun(\eC_1, \eC_2) \subset \Fun(\eC_1, \eC_2)$ the 
full subcategory spanned by left adjoint functors.

\item For a symmetric monoidal $\infty$-category $\eC$, we let $\CAlg(\eC)$ and $\Alg(\eC)$ denote the $\infty$-category of commutative algebras in $\eC$ and the $\infty$-category of associative algebras in $\eC$ respectively.

\end{itemize}

\end{NaC}

\section{Preliminaries on \texorpdfstring{$k$}{k}-linear stable \texorpdfstring{$\infty$}{infty}-categories}\label{sec:dg}

\subsection{Basic concepts in $\infty$-category theory}

In this section, we recall some basic concepts in $\infty$-category theory.
Standard references are \cite{HTT} and \cite{HA}.

An $\infty$-category $\eC$ with a zero object is called \textbf{stable} 
if it admits finite limits and colimits and a square in $\eC$ is a pushout square if and only if it is a pullback square.
The homotopy category of a stable $\infty$-category carries a natural triangulated structure.

An $\infty$-category $\eC$ is called \textbf{presentable} if it admits all small colimits and it satisfies a certain set-theoretic assumption called accessibility (see \cite[\S 5.5]{HTT} for the detail).
It is shown in \cite[Corollary 5.5.2.4]{HTT} that a presentable $\infty$-category admits all small limits.
The \textbf{adjoint functor theorem} \cite[Corollary 5.5.2.9]{HTT} states that a functor between presentable $\infty$-categories $F \colon \eC \to \eD$ has a right adjoint precisely when it preserves small colimits, and has a left adjoint precisely when it is accessible and it preserves small limits.
For presentable $\infty$-categories $\eC_1$ and $\eC_2$, the $\infty$-category $\LFun(\eC_1, \eC_2)$ is again presentable (see \cite[Proposition 5.5.3.8]{HTT}).

An object $x$ in an $\infty$-category $\eC$ is called \textbf{compact} if the functor 
\[
\Map_{\eC}(x, -) \colon \eC \to \eS
\]
preserves small filtered colimits.
An $\infty$-category $\eC$ is \textbf{compactly generated} if it admits small colimits and it is generated under colimits by a set of compact objects. Compactly generated $\infty$-categories are presentable.

We let $\PrL \subset \Cat$ (resp. $\PrR \subset \Cat$) denote the subcategory consisting of presentable $\infty$-categories and left (resp. right) adjoint functors.
We have a natural equivalence $(\PrL)^\mathrm{op} \simeq \PrR$. 
We let $\PrLc \subset \PrL$ (resp. $\PrRc \subset \PrR$) denote the subcategory consisting of compactly generated $\infty$-categories and left adjoint functors that preserve the compact objects (resp. right adjoint functors which preserves small filtered colimits).
The above equivalence restricts to $(\PrLc)^{\op} \simeq \PrRc$.
We let $\PrLst \subset \PrL$ denote the full subcategory spanned by presentable stable $\infty$-categories.
We define $\PrLstc \coloneqq \PrLc \cap \PrLst$.

In \cite[\S 4.8]{HA}, Lurie defines a natural closed symmetric monoidal structure on $\PrL$ with $\LFun$ being the internal hom, i.e., the tensor product is characterized by $\LFun(\eC\otimes \eD, \eE)\simeq \LFun(\eC, \LFun(\eD, \eE))$. The unit is the $\infty$-category $\eS$ of spaces.

It is shown in \cite[\S 4.8.2]{HA} that $\PrLst$ inherits a symmetric monoidal structure from that of $\PrL$.
The unit is the $\infty$-category $\Sp$ of spectra.
A compactly generated stable $\infty$-category is dualizable in $\PrLst$ (see \cite[Proposition D.7.2.3]{SAG}).

\subsection{$k$-linear stable $\infty$-categories}

Recall that we have fixed an algebraically closed field $k$ of characteristic zero.
We let $\Mod_k$ denote the derived $\infty$-category of differential graded $k$-modules.
The $\infty$-category $\Mod_k$ is stable and compactly generated, and its symmetric monoidal structure defines a commutative algebra object $\Mod_k \in \CAlg(\PrLstc)$.
We define $\Catk \coloneqq \Mod_{\Mod_k}(\PrLst)$ and call its object a \textbf{$k$-linear stable $\infty$-category} and its morphism a \textbf{$k$-linear functor}.
The $\infty$-category $\Catk$ inherits a symmetric monoidal closed structure from $\PrL$. 
Explicitly, the relative tensor product is the colimit of the simplicial object $\eC_1\otimes \Mod_k^{\otimes \bullet}\otimes \eC_2$ in $\PrL$, 
and the internal hom, which is denoted by $\LFun_k(\eC_1, \eC_2) $, is the limit of the cosimplicial object $\LFun(\eC_1\otimes \Mod_k^{\otimes\bullet}, \eC_2)\simeq \LFun(\eC_1, \LFun(\Mod_k^{\otimes\bullet}, \eC_2))$. 
The space of objects (the maximal sub $\infty$-groupoid) of $\LFun_k(\eC_1, \eC_2)$ is $\Map_{\Catk}(\eC_1, \eC_2)$.

We let $\Catkcg \coloneqq \Catk\times_{\PrLst} \PrLstc \subset \Catk$ denote the subcategory consisting of compactly generated $k$-linear stable $\infty$-categories and $k$-linear functors preserving compact objects. This category is equivalent to $\Mod_{\Mod_k}(\PrLstc)$.

The following statement is a consequence of \cite[Capter 1, 9.3.6 and 3.5.3]{GR17}:

\begin{prop}\label{prop:adjlinear}
Let $F \colon \eC \to \eD$ be a $k$-linear functor (in particular, it is a morphism of $\PrL$).
\begin{enumerate}
\item Assume that the right adjoint functor $F^{R}$ of $F$ preserves small colimits. Then $F^{R}$ has a natural $k$-linear structure.
\item Assume that $F$ has a left adjoint functor $F^{L}$. 
Then $F^{L}$ has a natural $k$-linear structure.
\end{enumerate}
\end{prop}

\subsection{Dg-categories and $k$-linear stable $\infty$-categories}

Here we recall the relation between small dg-categories and $k$-linear stable $\infty$-categories.

A \textbf{small dg-category} is a small category enriched over the symmetric monoidal category of complexes of $k$-vector spaces.
In other words, a small dg-category $\eA$ consists of a set of objects $\Obj(\eA)$, a dg-vector space of morphisms $\eA(a, a')$ for each $a, a' \in \Obj(\eA)$, and the identity and the composition maps satisfying the associativity relation.
A \textbf{dg-functor} $F \colon \eA_1 \to \eA_2$ between small dg-categories is given by a map of sets $F \colon \Obj(\eA_1) \to \Obj(\eA_2)$ and maps of dg-vector space 
$\eA_1(a, a') \to \eA_2(F(a), F(a'))$ preserving the identity and the composition maps.

For small dg-categories $\eA_1$ and $\eA_2$, the tensor product $\eA_1 \otimes_k \eA_2$ is defined as follows:
\begin{itemize}
    \item The set of objects $\mathrm{Obj}(\eA_1 \otimes_k \eA_2)$ is the product $\mathrm{Obj}(\eA_1) \times \mathrm{Obj}(\eA_2)$.
    \item For a pair of objects $(x, y), (x', y') \in \mathrm{Obj}(\eA_1 \otimes_k \eA_2)$, the morphism dg-vector space is defined by the tensor product
    \[
    \mathrm{Mor}((x, y), (x', y')) =  \mathrm{Mor}(x, y) \otimes_{k} \mathrm{Mor}(x', y').
    \]
    \item The composition is defined in the natural manner.
\end{itemize}
We can also define the internal mapping space $\Fun_{\dg}(\eA_1, \eA_2)$ as follows:
\begin{itemize}
    \item The set of objects $\Fun_{\dg}(\eA_1, \eA_2)$ is the set of dg-functors.
    \item For a pair of objects $F, G \in \mathrm{Obj}(\Fun_{\dg}(\eA_1, \eA_2))$, the morphism dg-vector space $\mathrm{Mor}(F, G)$ is defined as the dg-vector space of natural transformations.
    \item The composition is defined in the natural manner.
\end{itemize}

We let $\Mod_k^{\dg}$ the dg-category of complexes of $k$-vector spaces.
For a small dg-category $\eA$, we define
\[
\RMod^{\dg}_{\eA} \coloneqq \Fun_{\dg}(\eA^{\op}, \Mod_k^{\dg})
\]
and an object in $\RMod^{\dg}_{\eA}$ is called \textbf{right $\eA$-module}.
We let $\RMod_{\eA}^{\ord}$ the underlying ordinary category of $\RMod^{\dg}_{\eA}$.
The category $\RMod_{\eA}^{\ord}$ carries a combinatorial model structure induced by the projective model structure on $\Mod_{k}^{\ord}$.
The weak equivalences of right $\eA$-modules with respect to this model structure are the objectwise quasi-isomorphisms.
We define the derived $\infty$-category $\RMod_{\eA}$ by the underlying $\infty$-category of the model category $\RMod_{\eA}^{\ord}$, i.e., the $\infty$-categorical localization
\[
\mathcal{N}(\RMod_{\eA}^{\ord})[W^{-1}]
\]
where $\mathcal{N}$ denotes the ordinary nerve functor and $W$ is the class of quasi-isomorphisms.

Let $\eA_1$ and $\eA_2$ be small dg-categories.
A dg-functor $F \colon \eA_1 \to \eA_2$ is called \textbf{Morita equivalence} if the restriction functor
\[
F^* \colon \RMod_{\eA_2} \to \RMod_{\eA_1}
\]
defines an equivalence of $\infty$-categories.

For small dg-categories $\eA_1$ and $\eA_2$, we define the dg-category ${}_{\eA_1} \BMod_{\eA_2}^{\dg}$ of \textbf{$\eA_1$-$\eA_2$-bimodules} by
\[
{}_{\eA_1} \BMod_{\eA_2}^{\dg} \coloneqq \RMod_{\eA_1^{\op} \otimes_k \eA_2}^{\dg}.
\]
For an $\eA_1$-$\eA_2$-bimodule $M$ and objects $a_1 \in \eA_1$ and $a_2 \in \eA_2$,
we let ${}_{a_1} M_{a_2}$ the value of the functor $M$ at $a_1 \otimes a_2 \in \Obj(\eA_1^{\op} \otimes_k \eA_2)$.

Let $\eA_i$ ($i=1, 2, 3$) be a small dg-category and $M$ (resp. $N$) be an $\eA_1$-$\eA_2$-bimodule (resp. $\eA_2$-$\eA_3$-bimodule).
We define the $\eA_1$-$\eA_3$-bimodule $M \otimes_{\eA_2} N$ in a natural way such that for $a_1 \in \Obj(\eA_1)$ and $a_3 \in \Obj(\eA_3)$, the following equality holds:
\[
{}_{a_1}(M \otimes_{\eA_2} N)_{a_3} = \int^{a_2 \in \Obj(\eA_2)} ({}_{a_1}M_{a_2}) \otimes_k ({}_{a_2}N_{a_3}).
\]

Let $\dgcat_k^{\ord}$ be the ordinary category of small dg-categories whose morphisms are dg-functors.
In \cite{TabM} Tabuada constructed a model structure called \textbf{Morita model structure} on $\dgcat_k^{\ord}$ whose weak equivalences are Morita equivalences.
We let $\dgcat_k^{\mathrm{M}}$ denote the underlying $\infty$-category of this model category.

Consider the assignment 
\[
\eA \mapsto \RMod_{\eA}
\] 
for small dg-category $\eA$.
This assignment can be made $\infty$-functorial,
and it is shown in \cite[Corollary 5.7]{Coh13} that it induces an equivalence of $\infty$-categories
\begin{equation}\label{eq:dgsta}
\dgcat_k^{\mathrm{M}} \simeq \Catkcg
\end{equation}

Under this equivalence, an $\eA_1$-$\eA_2$-bimodule $X: \eA_1\to \RMod^{\dg}_{\eA_2}$ corresponds to a $k$-linear functor $\RMod_{\eA_1}\to \RMod_{\eA_2}$ via left Kan extension. 
The tensor product of bimodules corresponds to the composition of $k$-linear functors. 

\subsection{Monads and Barr--Beck theorem}

Let $\eC$ be a $k$-linear stable $\infty$-category.
The functor category $\LFun_k(\eC, \eC)$ carries a natural monoidal structure whose monoidal product is given by the composition. We let
\[
\Mnd_k(\eC) \coloneqq \Alg(\LFun_k(\eC, \eC))
\]
be the category of ($k$-linear) \textbf{monads} acting on $\eC$.
Note that the $\infty$-category $\eC$ is left-tensored over $\LFun_k(\eC, \eC)$.
Therefore for a monad $T \in \Mnd_k(\eC)$,
we can define the $\infty$-category $\LMod_T(\eC)$ of left $T$-modules in $\eC$.
It follows from \cite[Corollary 4.2.3.7(1)]{HA} that the $\infty$-category $\LMod_T(\eC)$ is presentable. One can see that it is also naturally a $k$-linear stable $\infty$-category.

A morphism of monads $T_1 \to T_2$ induces a forgetful functor
\[
\LMod_{T_2}(\eC) \to \LMod_{T_1}(\eC).
\]
It preserves limits and colimits by \cite[Corollary 4.2.3.7(2)]{HA}. One can also show that it has a natural $k$-linear structure.
Therefore, the adjoint functor theorem and Proposition \ref{prop:adjlinear} imply that we have a $k$-linear adjunction
\[\begin{tikzcd}
        \LMod_{T_2}(\eC) \arrow[r, shift right = .5ex]
        &  \LMod_{T_1}(\eC) \arrow[l, shift right = 1ex]
        \arrow[l, phantom, shift right = .2ex, "\scriptscriptstyle\boldsymbol{\bot}"]
  \end{tikzcd}.
\]

We collect some useful lemmas on monads and its modules:

\begin{lem}
Let $\eC$ be a compactly generated $k$-linear stable $\infty$-category and $T \in \Mnd_k(\eC)$ be a monad.
Then $\LMod_T(\eC)$ is compactly generated.
\end{lem}

\begin{proof}
We let $F \colon \LMod_{T}(\eC) \to \eC$ be the forgetful functor and $F^L \colon \eC \to \LMod_{T}(\eC)$ be its left adjoint.
Take a set of compact generators $\{ x_i \}_{i \in I}$ in $\eC$.
Since $F$ preserves filtered colimits, its left-adjoint $F^L$ preserves compact objects.
Therefore $\{ F^{L}( x_i) \}_{i \in I}$ is a set of compact objects in $\LMod_T(\eC)$.
Thus the claim follows from \cite[Corollary 2.5]{Yan21} and the fact that $F$ is conservative.
\end{proof}

\begin{lem}\label{lem:monadpush}
Let $\eC$ be a $k$-linear compactly generated stable $\infty$-category and $T_1 \to T_2$ and $T_1 \to T_3$ be morphisms of monads acting on $\eC$.
Write $T_4 \coloneqq T_2 \coprod_{T_1} T_3$.
Then the following diagram is a pushout square in $\Catkcg$:
 \[
 \xymatrix{
 \LMod_{T_1}(\eC) \ar[r] \ar[d] & \LMod_{T_2}(\eC) \ar[d] \\
 \LMod_{T_3}(\eC) \ar[r] & \LMod_{T_4}(\eC).
 }
\]
\end{lem}

\begin{proof}
As the forgetful functor $\Catkcg \to \PrLc$ is conservative and preserves small colimits,
it suffices to show that the above square is a pushout square in $\PrLc$, or equivalently, that the following is a pullback square in $\PrRc$:
\[
 \xymatrix{
 \LMod_{T_4}(\eC) \ar[r] \ar[d] & \LMod_{T_2}(\eC) \ar[d] \\
 \LMod_{T_3}(\eC) \ar[r] & \LMod_{T_1}(\eC).
 }
\]
Using \cite[Proposition 4.4.2.9]{HTT}, it suffices to prove that the above diagram is a pullback square in the category $\PrRc_{/\eC}$.

Let $U \colon \eD \to \eC$ be a morphism in $\PrRc$. 
We wish to show that the following diagram is a pullback square:
\begin{equation}\label{eq:square}
\xymatrix{
\Map_{\PrRc_{/\eC}}(\eD,  \LMod_{T_4}(\eC)) \ar[r] \ar[d] & \Map_{\PrRc_{/\eC}}(\eD,  \LMod_{T_2}(\eC)) \ar[d] \\
\Map_{\PrRc_{/\eC}}(\eD,  \LMod_{T_3}(\eC)) \ar[r] & \Map_{\PrRc_{/\eC}}(\eD,  \LMod_{T_1}(\eC))
}
\end{equation}

Now let $T \in \Alg(\Fun(\eC, \eC))$ be an arbitrary monad acting on $\eC$ and $T_U$ be the endomorphism monad of $U$ (see \cite[Definition 4.7.3.2, Proposition 4.7.3.3]{HA}).
By the definition of the endomorphism monad and \cite[Corollary 4.7.1.41]{HA},
we have an equivalence
\[
\Map_{\Alg(\Fun(\eC, \eC))}(T, T_U) \simeq \Map_{\Cat_{/ \eC}}(\eD, \LMod_{T}(\eC)).
\]
On the other hand, since the forgetful functor $\LMod_{T}(\eC)\to \eC$ is conservative and preserves small limits and small filtered colimits (so it reflects them), any functor $\eD\to \LMod_{T}(\eC)$ lifting $U$ preserves small limits and small filtered colimits, i.e., we have an equivalence
\[
\Map_{\Cat_{/ \eC}}(\eD, \LMod_{T}(\eC)) \simeq \Map_{\PrRc_{/ \eC}}(\eD, \LMod_{T}(\eC)).
\]
Thus the square \eqref{eq:square} is identified with the square
\[
\xymatrix{
\Map_{\Alg(\Fun(\eC, \eC))}(T_4, T_U) \ar[r] \ar[d] & \Map_{\Alg(\Fun(\eC, \eC))}(T_2, T_U) \ar[d] \\
\Map_{\Alg(\Fun(\eC, \eC))}(T_3, T_U) \ar[r] & \Map_{\Alg(\Fun(\eC, \eC))}(T_1, T_U),
}
\]
which is a pullback square, as the forgetful functor $\Mnd_{k}(\eC)\to \Alg(\Fun(\eC, \eC))$ preserves pushouts (notice that the monoidal functor $\LFun_k(\eC, \eC)\to \Fun(\eC, \eC)$ admits a right adjoint by \cite[Remark 5.5.2.10]{HTT}, which passes to the category of algebras). 
\end{proof}


\begin{lemdef}\label{lem:monadfree}
Let $\eC$ be a $k$-linear stable $\infty$-category and $F \in \LFun_k(\eC, \eC)$ be a $k$-linear functor.
Then there exists a \textbf{free monad} $TF \in \Mnd_k(\eC)$ with the following property:
\begin{itemize}
    \item For an arbitrary monad $M \in \Mnd_k(\eC)$, we have a natural equivalence
    \[
    \Map_{\Mnd_k(\eC)}(TF, M) \simeq \Map_{\LFun_k(\eC, \eC)}(F, T).
    \]
\item As a $k$-linear functor, we have an equivalence
\[
TF \simeq \Id_{\eC} \oplus F \oplus F^2 \oplus \cdots .
\]
\end{itemize}
\end{lemdef}

\begin{proof}
Note that the $\infty$-category $\LFun_k(\eC, \eC)$ has small colimits (which is computed pointwise) and its monoidal product preserves small colimits separately in each variable.
Therefore the claim follows from \cite[Proposition 4.1.1.18]{HA}.
\end{proof}

Now we recall Lurie's $\infty$-categorical version of the Barr--Beck theorem \cite[Theorem 4.7.3.5]{HA}.
Let $U \colon \eD \to \eC$ be a $k$-linear functor between $k$-linear stable $\infty$-categories which admits a left adjoint functor $U^L$. Then the endomorphism monad $T$ of $U$ (\cite[Definition 4.7.3.2]{HA}) exists, and its underlying functor is $U \circ U^{L}$.
One can show that $U$ naturally lifts to a $k$-linear functor $\bar{U} \colon \eD \to \LMod_T(\eC)$ such that the following diagram commutes:
\[
\xymatrix{
& \LMod_{T}(\eC) \ar[d] \\
\eD \ar[r]_-{U} \ar@{-->}[ru]^{\bar{U}} & \eC
}
\]
We say that $U$ is \textbf{monadic} if $\bar{U}$ is an equivalence of $\infty$-categories.

\begin{thm}[{Barr--Beck--Lurie, \cite[Theorem 4.7.3.5]{HA}}]
Let $U \colon \eD \to \eC$ be a $k$-linear functor which admits a left adjoint functor.
Then $U$ is monadic if and only if it is conservative.
\end{thm}

\subsection{Hochschild homology and negative cyclic homology}

Let $\eC$ be a compactly generated (hence dualizable) $k$-linear stable $\infty$-category.
Consider the following composition
\[
\Mod_k \xrightarrow[]{\coev_{\eC}} \eC \otimes \eC^{\vee} \simeq \eC^{\vee} \otimes \eC \xrightarrow[]{\ev_{\eC}} \Mod_k
\]
where $\coev_{\eC}$ denotes the coevaluation functor and $\ev_{\eC}$ denotes the evaluation functor.
We define the \textbf{Hochschild homology complex} of $\eC$, denoted by $\HH(\eC)$, to be the image of the one-dimensional vector space under the above composition.
It follows from \cite[Theorem 2.14]{HSS17} that the Hochschild homology complex $\HH(\eC)$ carries a natural $S^1$-action hence a mixed differential $\delta \colon \HH(\eC) \to \HH(\eC)[-1]$.
We define the \textbf{negative cyclic complex} $\HC^-(\eC)$ to be the homotopy $S^1$-fixed point of $\HH(\eC)$. 
It is shown in \cite[Theorem 2.14]{HSS17} that if we are given a $k$-linear functor between $k$-linear stable $\infty$-categories $F \colon \eC \to \eD$ which preserves the compact objects, we have a natural map $\HH(F) \colon \HH(\eC) \to \HH(\eD)$ of $S^1$-equivariant complexes.

\begin{ex}
Let $X$ be a smooth variety over $k$.
Then the HKR decomposition implies an isomorphism
\[
\HH_n(\Qcoh(X)) \simeq \bigoplus_{n = q - p} H^p(X, \wedge^q \Omega_X).
\]
If we are given a morphism between smooth varieties $f \colon X \to Y$, the map $\HH(f^*)$ is given by the pullback map of differential forms.
\end{ex}

For later use, we recall some basic facts about negative cyclic homology.
Firstly, by a generality of mixed complexes (see \cite[\S 2.5.13]{Lod98}) we obtain the Connes exact sequence
\[
\cdots \to \HC^{-}_{n+2}(\eC) \to \HC^{-}_{n}(\eC) \xrightarrow[]{\pi} \HH_n(\eC) \xrightarrow[]{\tilde{\delta}} \HC^{-}_{n+1}(\eC) \to \cdots
\]
where $\pi$ is the map forgetting the $S^1$-fixed point structure and $\tilde{\delta}$ is a natural lift of the mixed differential.
Now take $\eC = \Qcoh(X)$ for some $d$-dimensional smooth variety $X$. It is shown in \cite[Lemma 5.10]{BDI} that for $i > d$ the negative cyclic homology group $\HC^{-}_{i}(\Qcoh(X))$ vanishes.
Therefore we have natural isomorphisms
\begin{equation}\label{eq:HQcoh}
\HC^{-}_{d}(\Qcoh(X)) \cong \HH^{}_{d}(\Qcoh(X)) \cong H^0(X, \omega_X)
\end{equation}
and a natural short exact sequence
\begin{equation}\label{eq:HCQcoh2}
0 \to \HC^{-}_{d-1}(\Qcoh(X)) \to \HH_{d-1}(\Qcoh(X)) \xrightarrow[]{\delta} \HH_{d}(\Qcoh(X)).
\end{equation}

We now discuss Hochschild homology for a special class of $k$-linear stable $\infty$-categories called smooth. A $k$-linear stable $\infty$-category $\eC$ is called \textbf{smooth} if the evaluation map
\[
\ev_{\eC} \colon \eC^{\vee} \otimes \eC \to \Mod_{k}
\]
has a left adjoint $\ev_{\eC}^{L}$. In this case, we define the \textbf{inverse dualizing functor} $\Id_{\eC}^! \colon \eC \to \eC$ to be the image of the one-dimensional vector space under the functor
\[
\Mod_{k} \xrightarrow[]{\ev_{\eC}^{L}} {\eC}^{\vee} \otimes \eC \simeq \LFun_{k}(\eC, \eC).
\]
\begin{ex}\label{ex:invQcoh}
Let $X$ be a smooth variety over $k$.
Then $\Qcoh(X)$ is compactly generated $k$-linear stable $\infty$-category.
Further, it is shown in \cite[Proposition 4.7, Corollary 4.8]{BFN10} that we have equivalences
\[
\Qcoh(X \times X) \simeq \Qcoh(X)^{\vee} \otimes \Qcoh(X)  \simeq \LFun_{k}(\Qcoh(X), \Qcoh(X)) 
\]
which maps $\cP \in \Qcoh(X \times X)$ to the Fourier--Mukai transform with the kernel $\cP$.
Under this identification, the evaluation map is given by the composition
\[
\ev_{\Qcoh(X)} \colon \Qcoh(X \times X) \xrightarrow[]{\Delta^*} \Qcoh(X) \xrightarrow[]{\Gamma} \Mod_k
\]
where $\Delta \colon X \to X \times X$ is the diagonal map.
Using the Grothendieck duality, we see that the left adjoint to the evaluation functor is given by
\[
V \mapsto \Delta_*(a_X^* V \otimes \omega_X^{-1}[ - \dim X])
\]
where $a_X$ is the constant map from $X$ to a point. Therefore we have
\[
\Id^!_{\Qcoh(X)} = - \otimes \omega_X^{-1}[-\dim X].
\]
\end{ex}

By the definition of the inverse dualizing functor, we have a natural equivalence
\[
\HH(\eC) \simeq \Map_{\LFun_k}(\Id_{\eC}^!, \Id_{\eC}).
\]
Let $F \colon \eC \to \eD$ be a $k$-linear functor between smooth $k$-linear stable $\infty$-categories with a $k$-linear right adjoint $F^{R} \colon \eD \to \eC$. We can define a natural transformation
\begin{equation}\label{eq:desire}
    \eta \colon \Id_{\eD}^! \to F \circ \Id_{\eC}^! \circ F^R
\end{equation}
as follows: Firstly note that we have a natural transformation
\[
\alpha \colon \ev_{\eC} \to \ev_{\eD} \circ ((F^R)^{\vee} \otimes F)
\]
where $\ev_{\eC} \colon \eC^{\vee} \otimes \eC \to \Mod_k$ and $\ev_{\eD} \colon \eD^{\vee} \otimes \eD \to \Mod_k$ are evaluation functors.
Then consider the following composition
\begin{align*}
    \ev_{\eD}^L \to \ev_{\eD}^L  \circ \ev_{\eC} \circ \ev_{\eC}^L \xrightarrow{\ev_{\eD}^L \alpha \ev_{\eC}^L} \ev_{\eD}^L \circ \ev_{\eD} \circ ((F^R)^{\vee} \otimes F) \circ \ev_{\eC}^L \to ((F^R)^{\vee} \otimes F) \circ \ev_{\eC}^L.
\end{align*}
By evaluating this natural transformation at $k$, we obtain the desired map \eqref{eq:desire}.
It is shown in \cite[Proposition 4.4]{BDII} that the map $\HH(F)$ is given by the following composition
\[
\Hom(\Id_{\eC}^!, \Id_{\eC}) \to \Hom(F \circ \Id_{\eC}^! \circ F^R, F \circ F^{R}) \to \Hom(\Id_{\eD}^!, \Id_{\eD})
\]
where the latter map is defined using $\eta$ and the counit map.

\section{Deformed Calabi--Yau completion of \texorpdfstring{$k$}{k}-linear stable \texorpdfstring{$\infty$}{infty}-categories}\label{sec:defCY}

Let $\eA$ be a smooth small dg-category, $n$ be an integer and $c \in \HC^{-}_{n-1}(\eA)$ be a negative cyclic homology class.
In \cite{Kel11}, Keller introduced a new dg-category $\Pi_{n + 1}(\eA, c)$ called the \textbf{deformed $(n + 1)$-Calabi--Yau completion} of $\eA$.

In this section, we develop the theory of deformed Calabi--Yau completion using the language of $k$-linear stable $\infty$-category rather than small dg-category.
Our description of the Calabi--Yau completion is given by the $\infty$-category of left modules over a certain free monad. 
We will compare our description and Keller's original definition using the Barr--Beck--Lurie theorem.

\subsection{Tensor algebra category and deformed Calabi--Yau completion}\label{ssec:CYcomp}

Let $\eC$ be a $k$-linear stable $\infty$-category and $F \colon \eC \to \eC$ be a $k$-linear endofunctor.
As is shown in Lemma-Definition \ref{lem:monadfree}, we can construct the free monad $TF \in \Alg(\LFun_{k}(\eC, \eC))$ generated by $F$.
We define the \textbf{tensor algebra $\infty$-category} $T_{\eC}(F) \in \Catk$ by 
\[
T_{\eC}(F) \coloneqq \LMod_{TF}(\eC).
\]

Let us see that our definition of tensor algebra $\infty$-category is compatible with Keller's definition of tensor dg-category \cite[\S 4.1]{Kel11}. 
Let $\eA$ be a small dg-category and $X$ be a $\eA$-$\eA$-bimodule.
Consider the following $\eA$-$\eA$-bimodule
\[
T_{\eA}^{\dg}(X) \coloneqq \eA \oplus X \oplus (X \otimes_{\eA} X) \oplus \cdots.
\]
For objects $a_i \in \eA$ (i = 1, 2, 3),
we have a natural morphism of dg-vector spaces
\[
({}_{a_3}T_{\eA}^{\dg}(X)_{a_2}) \otimes_k ({}_{a_2}T_{\eA}^{\dg}(X)_{a_1}) \to {}_{a_3}T_{\eA}^{\dg}(X)_{a_1}.
\]
Therefore $T_{\eA}(X)$ carries a dg-category structure, whose set of objects are $\Obj(\eA)$ and the dg-vector space of morphisms is defined by 
\[
T_{\eA}(X)^{\dg}(a_1, a_2) = {}_{a_2}T_{\eA}^{\dg}(X)_{a_1}.
\]
The dg-category $T_{\eA}(X)$ is called \textbf{tensor dg-category}.

\begin{prop}\label{prop:vssmall}
Let $\eA$ be a small dg-category and $X$ be a cofibrant $\eA$-$\eA$-bimodule.
We let $F_X \colon \RMod_{\eA} \to \RMod_{\eA}$ be the $\infty$-functor given by tensoring $X$. Then we have an equivalence
\[
\RMod_{T_{\eA}^{\dg}(X)} \simeq T_{\RMod_{\eA}}(F_X).
\]
\end{prop}

\begin{proof}
Consider the forgetful functor $U \colon \RMod_{T_{\eA}^{\dg}(X)} \to \RMod_{\eA}$.
It admits a left adjoint $U^{L} = - \otimes_{\eA} T_{\eA}^{\dg}(X)$. 
Since $U$ is conservative and colimit preserving, the Barr--Beck--Lurie theorem implies that $U$ is monadic.
Therefore if we write $M \coloneqq U \circ U^L$ for the endomorphism monad, we need to show an equivalence of monads $M \simeq T F_X$.
This follows immediately from \cite[Proposition 4.1.1.18]{HA}.
\end{proof}

Now we introduce the notion of Calabi--Yau completion as a special case of tensor algebra $k$-linear $\infty$-category.

\begin{defin}[\cite{Kel11}]
Let $\eC$ be a smooth $k$-linear stable $\infty$-category.
For an integer $n$, we define the \textbf{$n$-Calabi--Yau completion} $\Pi_n(\eC)$ of $\eC$ by
\[
\Pi_n(\eC) \coloneqq T_{\eC}(\Id^!_{\eC}[n-1]).
\]
\end{defin}

It follows from Proposition \ref{prop:vssmall} that our definition of Calabi--Yau completion is compatible with Keller's definition in \cite[\S 4.1]{Kel11}.

\subsection{Deformed Calabi--Yau completion}\label{ssec:defCYcompt}

Let $\eC$ be a smooth $k$-linear stable $\infty$-category and take a Hochschild homology class $c \colon k[n-1] \to \HH(\eC) \simeq \Map(\Id_{\eC}^! , \Id_{\eC})$.
We let $i_c^{\#}$ be the morphism of monads $ T\Id_{\eC}^![n-1] \to \Id_{\eC}$ induced by $c$.
It defines a $k$-linear functor $\Pi_n(\eC) \to \eC$ which we denote by $i_c^*$.
We define the \textbf{deformed $(n+1)$-Calabi--Yau completion} $\Pi_{n+1}(\eC, c)$ of $\eC$ associated with $c$ by the following pushout square in $\Catkcg$:
\begin{equation}\label{eq:defCY}
    \xymatrix{
    \Pi_n(\eC) \ar[r]^-{i_0^*} \ar[d]^-{i_c^*} & \eC \ar[d] \\
    \eC \ar[r] & \Pi_{n+1}(\eC, c) \pocorner.
    }
\end{equation}
It follows from \cite[Proposition 5.5]{Kel11} that our definition of the deformed Calabi--Yau completion is compatible with Keller's definition in \cite[\S 5.1]{Kel11}.

\begin{rmk}
For a Hochschild homology class $c \colon k[n-1] \to \HH(\eC) $ we define the monad $T_c \Id_{\eC}^![n] \in \Alg(\LFun_{k}(\eC, \eC))$ by the following pushout square:
\[
    \xymatrix{
    T\Id_{\eC}^![n - 1] \ar[r]^-{i_0^{\#}} \ar[d]^-{i_c^{\#}} & \Id_{\eC} \ar[d] \\
    \Id_{\eC} \ar[r] & T_c\Id_{\eC}^![n] \pocorner.
    }
\]
Then it follows from Lemma \ref{lem:monadpush}
that we have an equivalence
\[
\Pi_{n+1}(\eC, c) \simeq \LMod_{T_c\Id_{\eC}^![n]}(\eC).
\]
When $c=0$, we have a natural equivalence $T_0\Id_{\eC}^![n] \simeq T\Id_{\eC}^![n]$. This is a consequence of the fact that the free monad functor
\[
\mathrm{Free} \colon \LFun_k(\eC, \eC) \to \Alg(\LFun_k(\eC, \eC))
\]
preserves colimits. Therefore we obtain an equivalence
\begin{equation}\label{eq:defundef}
\Pi_{n+1}(\eC, 0) \simeq \Pi_{n+1}(\eC). 
\end{equation}
\end{rmk}

\subsection{Relative Calabi--Yau structure}\label{ssec:CYstr}

We recall the notion of left Calabi--Yau structure and its relative version introduced in \cite{BDI}.

\begin{defin}
Let $\eC$ be a smooth $k$-linear stable $\infty$-category. A negative cyclic class $c \colon k[n] \to \HC^-(\eC)$ is called a \textbf{left Calabi--Yau structure} if its underlying Hochschild homology class $\HH_n(\eC) \simeq \Hom(\Id_{\eC}[n], \Id_{\eC}^!)$ induces an equivalence of functors
\[
\Id_{\eC}[n] \simeq \Id_{\eC}^!.
\]
\end{defin}

\begin{ex}\label{ex:CY}
Let $X$ be a smooth variety of dimension $n$.
Then it is shown in \cite[Lemma 5.12]{BDI} that an $n$-form
\[
c \in H^0(X, \omega_X) \cong \HH_n(\Qcoh(X)) \cong \HC^{-}_n(\Qcoh(X))
\]
corresponds to a left Calabi--Yau structures if it defines a trivialization of the line bundle $\omega_X$.
\end{ex}

Now we discuss the relative version.
\begin{defin}
Let $F \colon \eC \to \eD$ be a $k$-linear functor between smooth $k$-linear stable $\infty$-categories with a $k$-linear right adjoint $F^{r}$.
A \textbf{relative left $n$-Calabi--Yau structure} is a pair of a left $n$-Calabi--Yau structure $c \colon k[n] \to \HC^{-}(\eC)$ and a homotopy $\tau \colon \HC^{-}(F)(c) \simeq 0$ such that the map
\[
\mathrm{cofib}(\Id_{\eD}^![n] \to F \circ \Id_{\eC}^![n] \circ F^r) \to \Id_{\eD}
\]
induced by the map $F \circ \Id_{\eC}^![n] \circ F^r \xrightarrow{FcF^r}  F \circ F^r \to \Id_{\eD}$ and the null-homotopy $\tau$ is an equivalence.

\end{defin}

Let $F_i \colon \eC \to \eD_i$ be a $k$-linear functor between smooth $k$-linear stable $\infty$-categories with $k$-linear right adjoints for $i =1, 2$.
Write $\eE$ for the pushout of $F_1$ and $F_2$ i.e. $\eE$ fits into the following pushout square
\[
\xymatrix{
\eC \ar[r]^-{F_1} \ar[d]_-{F_2} & \eD_1 \ar[d] \\
\eD_{2} \ar[r] & \eE \pocorner.
}
\]
Assume that we are given relative $n$-Calabi--Yau structures for $F_1$ and $F_2$.
Then we obtain a loop at $0$ in the space $|\HC^{-}(\eE)[-n]|$ hence a map $\tilde{c} \colon k[n + 1] \to \HC^{-}(\eE)$. 
It is shown in \cite[Theorem 6.2]{BDI} that $\tilde{c}$ defines an $n+1$-Calabi--Yau structure on $\eE$.

\subsection{Calabi--Yau structure on the deformed Calabi--Yau completion}

It is shown in \cite{Kel18} that deformed Calabi--Yau completion of a small dg-category carries a left Calabi--Yau structure.
In this section, we will give a construction of the left Calabi--Yau structure using the language of $k$-linear stable $\infty$-categories
rather than small dg-categories as in the original paper \cite{Kel18}.

Let $\eC$ be a $k$-linear stable $\infty$-category and let $X \colon \eC \to \eC$ be a $k$-linear functor.
We let $\pi_* \colon T_{\eC}(X) \to \eC$ denote the forgetful functor and $\pi^* \colon  \eC \to T_{\eC}(X)$ denote its left adjoint functor.
We define the tautological map
\[
u \in \Map_{\LFun_{k}(T_{\eC}(X), T_{\eC}(X))}(\pi^* \circ X \circ \pi_*, \Id_{T_{\eC}(X)})
\]
by the composition
\[
\pi^* \circ X \circ \pi_* \xrightarrow[]{\pi^*(X \hookrightarrow \pi_* \pi^*)\pi_*} \pi^* \circ \pi_* \circ \pi^* \circ \pi_* \to \Id_{T_{\eC}(X)}
\]
where the latter map is defined by the counit map.
Now consider the case where $\eC$ is smooth and $X = \Id_{\eC}^![n-1]$.
We have seen in \eqref{eq:desire} that we have a natural transform 
\begin{equation}\label{eq:pullHoch}
\Id_{\Pi_n(\eC)}^! \to \pi^* \circ \Id_{\eC}^! \circ \pi_*.
\end{equation}
Now we define the \textbf{tautological Hochschild homology class} $\theta_{\Pi_{n}(\eC)} \colon k[n-1] \to \HH(\Pi_n(\eC))$ by the map corresponding to the following composition
\[
\Id_{\Pi_n(\eC)}^![n-1] \xrightarrow[]{\eqref{eq:pullHoch}} \pi^* \circ \Id_{\eC}^! [n-1] \circ \pi_* \xrightarrow[]{u} \Id_{\Pi_n(\eC)}.
\]
It is clear that this Hochschild homology class is naturally equivalent to the one constructed in \cite[Theorem 1.1]{Kel18}.
Therefore it follows that $\tilde{\delta} \theta_{\Pi_{n}(\eC)} \colon k[n] \to \HC^{-}(\eC)$ defines a left $n$-Calabi--Yau structure on $\Pi_n(\eC)$ which will be denoted by $\eta_{\Pi_n(\eC)}$.

\begin{ex}\label{ex:localmain}
Let $Y = \Spec A$ be a smooth affine scheme of dimension $n$ over $k$ and take $\eC = \Qcoh(Y)$.
Write $X \coloneqq \Tot_{Y}(\omega_Y)$.
Then Proposition \ref{prop:vssmall} and Example \ref{ex:invQcoh} implies an equivalence
\[
\Pi_{n + 1}(\Qcoh(Y)) \simeq \Qcoh(X).
\]
Under this equivalence and the HKR decomposition, the tautological Hochschild homology class $k[n] \to \HH(\Pi_{n+1}(\Qcoh(Y)))$
corresponds to the tautological $n$-form on $X$.
Therefore the left $n+1$-Calabi--Yau structure on $\Pi_{n+1}(\Qcoh(Y))$ corresponds to the natural Calabi--Yau form on $X$.

\end{ex}

Now we discuss the deformed case.
Let $\eC$ be a smooth $k$-linear stable $\infty$-category and take a Hochschild homology class $c \colon k[n-1] \to \HH(\eC)$ with a negative cyclic lift $\tilde{c}$.
Recall that the class $c$ induces a $k$-linear functor $i_c^* \colon \Pi_n(\eC) \to \eC$.
Clearly, we have a natural homotopy
\[
\HH(i_c^*)(\theta_{\Pi_n(\eC)}) \sim c.
\]
Therefore we obtain a natural homotopy
\[
\HC^{-}(i_c^*)(\eta_{\Pi_n(\eC)}) \sim \tilde{\delta} c \sim 0.
\]
where the latter homotopy is defined by the negative cyclic lift $\tilde{c}$.
It is shown in \cite[Proposition 5.17]{BCS20} that the above homotopy defines a relative Calabi--Yau structure on $i_c^*$.
Therefore using the result of \S \ref{ssec:CYstr}, we see that the deformed Calabi--Yau completion $\Pi_n(\eC, \tilde{c}) \coloneqq \Pi_n(\eC, c)$ carries a natural left Calabi--Yau structure denoted by $\eta_{\Pi_n(\eC, \tilde{c})}$ (depending on $\tilde{c}$).

\begin{rmk}\label{rmk:defundef}
Recall that we have seen in \eqref{eq:defundef} that there exists a natural equivalence $\Pi_n(\eC, 0) \simeq \Pi_n(\eC)$. This equivalence identifies left Calabi--Yau structures $\eta_{\Pi_n(\eC, 0)}$ and $\eta_{\Pi_n(\eC)}$.
\end{rmk}

\section{Deformed Calabi--Yau completion and torsor}\label{sec:main}

The aim of this section is to prove Theorem \ref{thm:intromain}.
To do this, we first describe the $n$-Calabi--Yau completion of $\Qcoh(Y)$ for smooth variety $Y$ in \S \ref{ssec:CYcompQ}.
We deduce the deformed version from the undeformed version in \S \ref{ssec:dCYcomp}.

\subsection{Calabi--Yau completion of $\Qcoh(Y)$}\label{ssec:CYcompQ}

Let $Y$ be a smooth variety of dimension $d$.
We let $\eA_n \in \Alg(\Qcoh(Y))$ be the free algebra generated by $\omega_Y^{-1}[-n]$. We prove the following statement which generalizes a result of Ikeda--Qiu \cite[\S 2.5]{IQ18}:

\begin{prop}\label{prop:CYcomp}
We have an equivalence of $k$-linear stable $\infty$-categories
\[
\Pi_n(\Qcoh(Y)) \simeq \RMod_{\eA_{n-d-1}}(\Qcoh(Y)).
\]

\begin{proof}
Let $U \colon \RMod_{\eA_{n-d-1}}(\Qcoh(Y)) \to \Qcoh(Y)$ be the forgetful functor. The functor $U$ admits a left adjoint $U^L$ given by tensoring $\eA_{n-d-1}$.
It is clear that $U$ is conservative and colimit preserving.
Therefore $U$ is monadic.
Now we need to show that the monad $M = U \circ U^L$ is equivalent to the monad $T_{\Id_{\Qcoh(X)}^![n-1]}$.
This follows immediately from Example \ref{ex:invQcoh} and \cite[4.1.1.18]{HA}.
\end{proof}

\end{prop}

\begin{cor}[{\cite[\S 2.5]{IQ18}}]
Write $X \coloneqq \Tot_Y(\omega_Y)$.
Then we have an equivalence of $k$-linear stable $\infty$-categories
\begin{equation}\label{eq:IQ}
\Pi_{d+1}(\Qcoh(Y)) \simeq \Qcoh(X).
\end{equation}
Further, this equivalence identifies the left Calabi--Yau structure $\eta_{\Pi_{d+1}(\Qcoh(Y))}$ and the Calabi--Yau structure on $\Qcoh(X)$ induced by the natural Calabi--Yau form on $X$.
\end{cor}

\begin{proof}
The equivalence \eqref{eq:IQ} follows immediately from \cite[Proposition 6.3.4.6]{SAG} and Proposition \ref{prop:CYcomp}.
To prove that this equivalence preserves the left Calabi--Yau structure, we may assume $Y$ is affine since the negative cyclic homology group $\HC^{-}_{d + 1}(\Qcoh(X)) \cong H^0(X, \omega_X)$
satisfies the sheaf property.
In this case the claim follows from Example \ref{ex:localmain}.
\end{proof}

\subsection{Deformed Calabi--Yau completion of $\Qcoh(Y)$}\label{ssec:dCYcomp}

Let $Y$ be a smooth variety of dimension $d$.
Take a cohomology class $c \in H^1(Y, \omega_Y)$.
Using the natural inclusion $ H^1(Y, \omega_Y) \hookrightarrow \HH_{d-1}(\Qcoh(Y))$, we can regard $c$ as a $(d-1)$-th Hochschild homology class.
Further, the short exact sequence \eqref{eq:HCQcoh2} implies that $c$ naturally lifts to a negative cyclic class $\tilde{c}$.

Let $X$ be the total space of the $\omega_Y$-torsor corresponding to $c$.
Then we have the following statement:

\begin{thm}\label{thm:defCYcomp}
We have an equivalence of $k$-linear stable $\infty$-categories
\begin{equation}\label{eq:defCYcomp}
\Pi_{d + 1}(\Qcoh(Y), \tilde{c}) \simeq \Qcoh(X).
\end{equation}
Further, this equivalence identifies the left Calabi--Yau structure $\eta_{\Pi_{d + 1}(\Qcoh(Y))}$ and the Calabi--Yau structure on $\Qcoh(X)$ induced by the natural Calabi--Yau form on $X$.
\end{thm}

\begin{proof}
It follows from the diagram \eqref{eq:defCY} and Proposition \ref{prop:CYcomp} that we have the following pushout square:
\[
    \xymatrix{
    \RMod_{\eA_1}(\Qcoh(Y)) \ar[r]^-{i_0^*} \ar[d]^-{i_c^*} & \Qcoh(Y) \ar[d] \\
    \Qcoh(Y) \ar[r] & \Pi_{d+1}(\Qcoh(Y), c) \pocorner.
    }
\]
Now define an algebra object $\eA_{0, c} \in \Alg(\Qcoh(Y))$ by the pushout
\[
    \xymatrix{
    \eA_1 \ar[r]^-{i_0^{\#}} \ar[d]^-{i_c^{\#}} & \cO_Y \ar[d] \\
    \cO_Y \ar[r] & \eA_{0, c} \pocorner.
    }
\]
Then it follows from \cite[Corollary 4.8.5.13]{HA} that there exists an equivalence
\[
\Pi_{d+1}(\Qcoh(Y), c) \simeq \RMod_{\eA_{0, c}}(\Qcoh(Y)).
\]

Let $\pi \colon X \to Y$ be the natural projection.
Then it is enough to show the following equivalence in $\Alg(\Qcoh(Y))$
\[
\eA_{0, c} \simeq \pi_* \cO_X.
\]
To do this, consider the following diagram
\[
    \xymatrix{
    X \ar[r]^-{\pi}  \ar[d]^-{\pi} \pbcorner & Y \ar[d]^-{0} \\
    Y \ar[r]^-{c} & [Y / \omega_Y] .
    }
\]
Here $[Y / \omega_Y]$ is the quotient stack by the trivial action of $\Tot_Y(\omega_Y)$ on $Y$.
By pushing down the structure sheaves of schemes and stacks appearing in the above diagram to $Y$, we obtain the following commutative diagram in $\Alg(\Qcoh(Y))$:
\[
\xymatrix{
\Sym_{\cO_Y}(\omega_Y^{-1}[-1]) \ar[r] \ar[d] & \cO_Y \ar[d] \\
\cO_Y \ar[r] & \pi_* \cO_X.
}
\]
By precomposing the natural map $\eA_1 \to \Sym_{\cO_Y}(\omega_Y^{-1}[-1])$, we obtain the following diagram
\[
\xymatrix{
\eA_1 \ar[r] \ar[d] & \cO_Y \ar[d] \\
\cO_Y \ar[r] & \pi_* \cO_X.
}
\]
To prove the equivalence \eqref{eq:defCYcomp} it is enough to show that this diagram is a pushout square.
To do this, we may assume that $Y$ is affine in which case $c$ is automatically equivalent to zero.
In this case, the algebra $\pi_* \cO_X$ is the free algebra generated by $\omega_Y^{-1}$. Therefore the claim follows from the fact that the formation of the free algebra commutes with the pushout.

Now we need to show that the equivalence \eqref{eq:defCYcomp} preserves left Calabi--Yau structure.
To do this, we may assume $Y$ is affine.
In this case the cohomology class $c \in H^1(Y, \omega_Y)$ automatically vanishes.
Then the claim follows from Remark \ref{rmk:defundef} and Example \ref{ex:localmain}.

\end{proof}

\section{Application to local curves}\label{sec:App}\label{sec:app}

We now apply Theorem \ref{thm:defCYcomp} to the geometry of the derived moduli stack of coherent sheaves on local curves.
Here \textbf{local curve} is a Calabi--Yau threefold of the form $\Tot_C(N)$ where $C$ is a smooth projective curve and $N$ is a rank two vector bundle with $\det(N) \cong \omega_C$.

\subsection{Shifted symplectic structure}

Here we briefly recall the theory of shifted symplectic structure introduced in \cite{PTVV13}.
Let $X$ be a derived Artin stack over $k$.
Define the space of \textbf{$n$-shifted $p$-forms} on $X$ by
\[
\cA^p(X, n) \coloneqq \Gamma(X, \wedge^2 \bL_X[n]).
\]
where $\bL_X$ is the cotangent complex of $X$.
We can also define the space of \textbf{$n$-shifted closed $p$-forms} $\cA^{p, \cl}(X, n)$. See \cite[Definition 1.12]{PTVV13} for the detail.
Roughly speaking, an $n$-shifted closed $p$-form is given by 
an $n$-shifted $p$-form $\omega_0$, a homotopy $\omega_1 \colon \ddr \omega_0 \sim 0$, a homotopy $\ddr \omega_2 \colon d\omega_1 \sim 0$ and so on.
We have a forgetful map
\[
\pi \colon \cA^{p, \cl}(X, n) \to \cA^{p}(X, n)
\]
and the de Rham differential map
\[
\ddr \colon \cA^{p}(X, n) \to \cA^{p+1, \cl}(X, n).
\]

\begin{defin}
Let $X$ be a derived Artin stack.
An \textbf{$n$-shifted symplectic structure} is an $n$-shifted closed $2$-form $\omega$ on $X$ whose underlying 
$n$-shifted $2$-form $\omega_0$ induces an equivalence
\[
\bL_{X}^{\vee} \xrightarrow[]{\simeq} \bL_{X}[n].
\]
\end{defin}

\begin{ex}
Let $X$ be a derived Artin stack and $f \in \Gamma(X, \cO_X)$ be a function on it.
The derived critical locus $\Crit(f)$ is given by the following Cartesian square
\[
\xymatrix{
\Crit(f) \pbcorner \ar[r] \ar[d] & X \ar[d]^-{0} \\
X \ar[r]^-{\ddr f} & \Tot_{X}(\bL_X)
}
\]
Then it follows from \cite[\S 4.2.1]{BCS20} that $\Crit(f)$ carries a canonical $-1$-shifted symplectic structure.
\end{ex}

\subsection{Moduli of objects in a Calabi--Yau category}

Let $\eC$ be a compactly generated $k$-linear stable $\infty$-category. 
Following \cite{TV07}, we define the moduli of objects in $\eC$ to be the prestack $\fM_C$ whose value on a commutative dg-algebra $A$ is the space
\[
\fM_{\eC}(A) = \LFun_{k}^{\omega}(C, \Mod_{A})
\]
where $\LFun_{k}^{\omega}$ denotes the space of $k$-linear functors which preserves compact objects.
Assume that $\eC$ is \textbf{of finite type}, i.e., a compact object in the $\infty$-category $\Catkcg$ of compactly generated $k$-linear stable $\infty$-category with $k$-linear functors which preserves compact objects.
Then it is shown in \cite[Proposition 2.14, Theorem 3.6]{TV07} respectively that $\eC$ is smooth and the prestack $\fM_{\eC}$ is locally an Artin stack of finite presentation over $k$.

\begin{thm}[{\cite[Theorem 5.5]{BDII}}]
Let $\eC$ be a finite type $k$-linear stable $\infty$-category and $c \colon k[n] \to \HC^{-}(\eC)$ be a left Calabi--Yau structure.
Then $\fM_{\eC}$ carries a natural $(2-n)$-shifted symplectic structure.
\end{thm}

\begin{ex}
Let $X$ be a smooth variety of dimension $n$.
In this case, the moduli space $\fM_X \coloneqq \fM_{\Qcoh(X)}$ of objects in $\Qcoh(X)$ parametrizes compactly supported perfect complexes on $X$.
Assume that $X$ carries a Calabi--Yau form.
As we have seen in Example \ref{ex:CY}, the Calabi--Yau form induces a left $n$-Calabi--Yau structure on $\Qcoh(X)$ and the above theorem implies 
that $\fM_X$ carries a natural $(2-n)$-shifted symplectic structure.
\end{ex}

Let $\eC$ be a $k$-linear stable $\infty$-category of finite type and $c \colon k[n-1] \to \HC^{-}(\eC)$ be a negative cyclic homology class. Then it is shown in \cite[\S 3.3]{Yeu16} that $\Pi_{n+1}(\eC, c)$ is of finite type. 

The following theorem, which is a special case of \cite[Corollary 6.19]{BCS20}, describes the moduli space of objects in a certain deformed $3$-Calabi--Yau completion:

\begin{thm}[{\cite[Corollary 6.19]{BCS20}}]\label{thm:BCS}
Let $\eC$ be a $k$-linear stable $\infty$-category of finite type and $ c \colon k \to \HH(\eC)$ be a zeroth Hochschild homology class. Let $\delta \colon \HH(\eC)[1] \to \HC^-(\eC)$ be the map induced from the mixed differential. Then there exists a natural function $f$ on $\fM_{\eC}$ such that there exists a natural equivalence of $-1$-shifted symplectic derived locally Artin stacks
\[
\fM_{\Pi_3(\eC, \delta c)} \simeq \Crit(f).
\]
\end{thm}

For a point $x \in \fM_{\eC}$, the value of the function $f$ is computed as follows:
Recall that $x$ corresponds to a $k$-linear functor $x \colon \eC \to \Mod_k$ which preserves compact objects.
Therefore we obtain a map
\[
\HH(\eC) \to \HH(\Mod_k) \simeq k.
\]
Then the value $f(x) \in k$ is given by the image of $c$ under the above map.

\subsection{Local curve as deformed Calabi--Yau completion}

Let $C$ be a smooth projective curve and $N$ be a rank two vector bundle on $C$ with an isomorphism $\det(N) \cong \omega_C$.
Assume that we are given a short exact sequence
\[
0 \to L_1 \to N \to L_2 \to 0
\]
where $L_1$ and $L_2$ are line bundles.
Write $X \coloneqq \Tot_C(N)$ and $Y \coloneqq \Tot_C(L_2)$.
Then we have the following statement:

\begin{thm}\label{thm:crit}
There exists a function $f$ on the derived moduli stack $\fM_{Y}$ of compactly supported prefect complexes on $Y$ such that there exists a natural equivalence of $-1$-shifted symplectic derived Artin stacks
\[
\fM_{X} \simeq \Crit(f). 
\]

\begin{proof}
Note that $X$ is the total space of an $\omega_Y$-torsor.
We let $c \in H^1(Y, \omega_Y)$ be the class corresponding to the torsor.
Then Theorem \ref{thm:defCYcomp} implies an equivalence of left $3$-Calabi--Yau categories
\[
\Qcoh(X) \simeq \Pi_3(\Qcoh(Y), c).
\]
Now using Theorem \ref{thm:BCS} and the injectivity of the map $\HC^{-}_1(\Qcoh(Y)) \to \HH_1(\Qcoh(Y))$ we have seen in \eqref{eq:HCQcoh2}, it is enough to prove that there exists a Hochschild homology class $c' \in \HH_0(\Qcoh(Y))$ such that we have an equality of the Hochschild homology classes $c = \delta c'$.

Let $\alpha \in \Ext^1(L_2, L_1) \cong H^1(C, \omega_C \otimes L_2^{\otimes^{-2}})$ the class corresponding to the extension $0 \to L_1 \to N \to L_2 \to 0$.
Let $\pi \colon X \to Y$ and $p \colon Y \to C$ be the projection.
Consider the following composition:
\[
\tau \colon H^1(C, \omega_C \otimes L_2^{\otimes^{-2}})
\to H^1(Y, p^* \omega_C) \to H^1(Y, \Omega_Y)
\subset \HH_0(\Qcoh(Y))
\]
Here the first map is defined using the inclusion $\omega_C \otimes L_2^{\otimes^{-2}} \hookrightarrow p_* p^* \omega_C$.
We set $c' \coloneqq 1/2 \cdot \tau(\alpha)$.
Now we claim the equality $c = \delta c'$.
To do this, we first describe the class $c$.
Consider the following short exact sequence:
\[
0 \to \mcO_Y \to \pi_* \mcO_X \to \pi_* \mcO_X/ \mcO_Y \to 0.
\]
This gives a map $\pi_* \mcO_X/ \mcO_Y \to \cO_Y[1]$.
Precomposing the natural map $p^* L_1^{-1} \to \pi_* \mcO_X/ \mcO_Y$ to this, we obtain a map $p^* L_1^{-1} \to \cO_Y[1]$.
This map corresponds to the class $c$.
Note that we have the following map between short exact sequences
\[
\xymatrix{
0 \ar[r] & L_2^{-1} \ar[r] \ar[d] & N^\vee \ar[r] \ar[d] & L_1^{-1} \ar[r] \ar[d] & 0 \\
0 \ar[r] & p_* \mcO_Y \ar[r] & p_* \pi_* \mcO_X \ar[r] & p_*(\pi_* \mcO_X/ \mcO_Y) \ar[r]  & 0.
}
\]
Therefore the composition
\[
p^* L_1^{-1} \xrightarrow[]{p^* \alpha} p^* L_2^{-1}[1] \to  \mcO_Y[1],
\]
corresponds to the class $c$.
Now it is enough to show that $\delta c'$ also corresponds to the above map.
Consider the following diagram:
\[
\xymatrix{
p_* p^* \omega_C \ar[r] & p_* \Omega_Y \ar[r]^-{d_{\mathrm{DR}}} & p_*\omega_Y \ar[r]^-{\sim} & p_* p^* L_1  \\
\omega_C \otimes L_2^{\otimes ^{-2}} \ar[r]^-{\cdot 2} \ar@{^{(}->}[u]  & \omega_C \otimes L_2^{\otimes ^{-2}} \ar[rr]^-
{\sim} & & L_1 \otimes L_2^{{-1}} \ar@{^{(}->}[u].
}
\]
It is clear that this diagram commutes.
Then the claim follows from that fact that the mixed differential corresponds to the de Rham differential under the HKR isomorphism.

\end{proof}

We now describe the function $f$ in the above theorem when we have $\deg(L_2) > 2g(C) - 2$.

\begin{prop} \label{prop:potential}
Let $(E, \phi)$ be a pair of a vector bundle $E$ on $C$ and a map $\phi \colon E \to E\otimes L_2$.
The value of the function $f \colon \fM_Y \to \bA^1$ introduced in Corollary \ref{thm:crit} at $[(E, \phi)] \in \fM_Y$ is
\[
(1/2) \cdot \alpha \cdot \Tr(\phi^2) \in k,
\]
where $\alpha \in \mH^0(C, L_2^{\otimes 2})^\vee \cong \Ext^1(L_2, L_1)$ is the extension class of the short exact sequence $0 \to L_1 \to N \to L_2 \to 0$.
\end{prop}

\end{thm}

\begin{proof}
Let $(\mathrm{gr}_{\mathrm{HN}}(E), (\mathrm{gr}_{\mathrm{HN}}(\phi))$ be the associated graded of the Harder--Narashimhan filtration of $(E, \phi)$.
It is clear that the value of $f$ at $[(E, \phi)]$ and $[(\mathrm{gr}_{\mathrm{HN}}(E), (\mathrm{gr}_{\mathrm{HN}}(\phi))]$ is equal.
Therefore we may assume that $(E, \phi)$ is a direct sum of semistable objects.
Using the additivity of $f$, we may further assume $(E, \phi)$ is semistable.
Furthermore, the local irreducibility of $\fM_Y^{\mathrm{ss}}$ proved in \cite[Proposition 2.9]{MS20} implies that the locus $\mcU \subset \fM_Y^{\mathrm{ss}}$ consisting of sheaves with smooth supports on $S$ is dense in $\fM_Y^{\mathrm{ss}}$.
Therefore we may assume that the one-dimensional sheaf $\widetilde{E}$ on $Y$ corresponding to $(E, \phi)$ is of the form $i_* L$, where $i \colon \widetilde{C} \hookrightarrow Y$ is a closed immersion from a smooth projective curve and $L$ is a line bundle on $\widetilde{C}$. 

What we need to compute is the value at $c \in \HH_0(\Qcoh(Y))$ of the function
\[
 \HH_0(\Qcoh(Y)) \to \HH_0(\Mod_k) \cong k
\]
induced by the functor $\Hom_{Y}(i_*L, -)$.
Since tensoring line bundles on $\widetilde{C}$ acts trivially on $\mH^0(\widetilde{C}, \mcO_{\widetilde{C}}) \subset \HH_0(\widetilde{C})$,
we may assume $L = \omega_{\widetilde{C}}$ and what we need to compute is the image of $c \in \HH_0(\Qcoh(Y))$ under the map
\[
\HH_0(\Qcoh(Y)) \xrightarrow[]{i^*} \HH_0(\Qcoh(\widetilde{C})) \xrightarrow[]{(\widetilde{C} \to \Spec k)_*} k. 
\]
Here the class $c'$ is defined in the proof of the previous theorem.
It follows from \cite[Theorem 2.21, Remark 2.22]{BN13} that the composite of the inclusion $\mH^1(Y, \Omega_Y) \hookrightarrow \HH_0(\Qcoh(Y))$ and the above map is given by the following composition:
\[
\mH^1(Y, \Omega_Y) \xrightarrow[]{i^*} \mH^1(\widetilde{C}, \omega_{\widetilde{C}}) \xrightarrow[]{\int_{\widetilde{C}}} k.
\]

Now consider the following diagram:
\[
\xymatrix{
\pi_* \pi^* \omega_C \ar[r] & \pi_* i_* i^* \pi^* \omega_C \ar[r] & \pi_* i_* \omega_{\widetilde{C}} \ar[d]  \\
\omega_C \otimes L_2^{\otimes^{-2}} \ar@{^{(}->}[u] \ar[rr]^-{\Tr(\phi^2)} & &  \omega_C. }
\]

It is clear that the above diagram commutes.
This implies that the following diagram commutes:
\[
\xymatrix{
\mH^1(Y, \pi^* \omega_C) \ar[r]& \mH^1(Y, \Omega_Y) \ar[r]^-{i^*} & \mH^1(\widetilde{C}, \omega_{\widetilde{C}}) \ar[r]_-{\sim}^-{\int_{\widetilde{C}}}  \ar[d] & k \ar@{=}[d] \\
\mH^1(C, \omega_C \otimes L_2^{\otimes^{-2}}) \ar[rr]^-{\cdot \Tr(\phi^2)} \ar@{^{(}->}[u] & & \mH^1(C, \omega_C) \ar[r]^-{\int_{C}}_-{\sim} & k
}
\]

Thus we obtain the desired identity
\[
\int_{\widetilde{C}} (i^* c') = (1/2) \cdot \alpha \cdot \Tr(\phi^2).  
\]

\end{proof}

\newcommand{\etalchar}[1]{$^{#1}$}
\providecommand{\bysame}{\leavevmode\hbox to3em{\hrulefill}\thinspace}
\providecommand{\MR}{\relax\ifhmode\unskip\space\fi MR }
\providecommand{\MRhref}[2]{%
  \href{http://www.ams.org/mathscinet-getitem?mr=#1}{#2}
}
\providecommand{\href}[2]{#2}

\end{document}